\documentclass[preprint,12pt]{elsarticle}
\usepackage{amssymb}
\usepackage{amsmath}
\usepackage{color}
\usepackage{amsthm}
\newtheorem{lm}{{\bf {Lemma}}}[section]
\newtheorem{thm}{{\bf {Theorem}}}[section]

\newtheorem{remark}{{\bf {Remark}}}[section]

\begin{document}
\begin{frontmatter}

\title{Small limit cycles bifurcating in pendulum
systems under trigonometric perturbations}
\author[add1]{Yun Tian\corref{cor1}}
\ead{ytian22@shnu.edu.cn}

\author[add1]{Tingting Jing}
\author[add1]{Zhe Zhang}

\cortext[cor1]{Corresponding author.}
\address[add1]{Department of Mathematics, Shanghai Normal University, Shanghai, China}

\begin{abstract}
In this paper, we consider the bifurcation of small-amplitude limit cycles near the origin
in perturbed pendulum systems of the form
$\dot x= y$, $\dot y=-\sin(x)+\varepsilon Q(x,y)$,
where $Q(x,y)$ is a smooth or piecewise smooth polynomial in the triple $(\sin(x), \cos(x), y)$ with free coefficients.
We obtain the sharp upper bound on the number of positive zeros of its associated first order Melnikov function near $h=0$
for $Q(x,y)$ being smooth and piecewise smooth with the discontinuity at $y=0$,
respectively.
\end{abstract}

\end{frontmatter}

\section{Introduction and main results }

Consider cylinder autonomous systems of the form
\begin{equation}\label{h1}
\dot x = H_y(x,y) + \varepsilon P(x,y),\quad
\dot y = -H_x(x,y) + \varepsilon Q(x,y),
\end{equation}
where $|\varepsilon|\ll1$, $H$, $P$ and $Q$ are analytic functions in $x$ and $y$ satisfying
\begin{equation*}
H(x+\omega,y)=H(x,y),\quad
P(x+\omega,y)=P(x,y),\quad
Q(x+\omega,y)=Q(x,y)
\end{equation*}
for a constant $\omega>0$.
System \eqref{h1} can be taken as a $\omega$-periodic perturbation of the Hamiltonian system
\begin{equation}\label{h2}
\dot x= H_y(x,y),\quad \dot y=-H_x(x,y),
\end{equation}
on the circular cylinder $[0,\omega]\times\mathbb{R}$.

As we know, system \eqref{h2} on the cylinder could have two types of periodic orbits:
oscillatory type and rotary type.
An orbit $\Gamma_h$ of system \eqref{h2} is called a periodic orbit of oscillatory type
if the solutions $(x(t),y(t))$ corresponding to $\Gamma_h$
satisfy
\[x(t+T_1)=x(t),\quad y(t+T_1)=y(t)\]
 for a positive constant $T_1$.
If the corresponding solutions $(x(t),y(t))$ of $\Gamma_h$
satisfy
\begin{equation*}
\begin{split}
&x(t+T_2)=x(t)+\omega,\quad y(t+T_2)=y(t),\\
\mbox{or}\,\, &x(t+T_2)=x(t)-\omega,\quad y(t+T_2)=y(t)
\end{split}
\end{equation*}
  for a constant $T_2>0$,
then $\Gamma_h$ is called a periodic orbit of rotary type.
Clearly,
periodic orbits of oscillatory type can be continuously deformed to a point,
while periodic orbits of rotary type surround the circular cylinder, and can not be
continuously deformed to a point.
The region continuously filled with periodic orbits of oscillatory (rotary) type
is the so-called oscillatory (rotary) region.

Under small perturbations, limit cycles can be produced in system \eqref{h1}
from a oscillatory (rotary) region $\Sigma=\bigcup_{h\in J}\Gamma_h$,
where $\Gamma_h$ represents a periodic orbit defined by $H(x,y)=h$,
and $J$ is an open interval.
We can see that the number of limit cycles bifurcating from $\Sigma$ for $\varepsilon$ sufficiently small
can be estimated by the maximum number of isolated zeros in $J$ (counting multiplicity) of
the first order Melnikov function given by
\begin{equation*}
M(h)=\oint_{\Gamma_h}Q(x,y)\mathrm{d}x-P(x,y)\mathrm{d}y, \quad h\in J.
\end{equation*}
Determining the maximum number of isolated zeros of $M(h)$ can be seen
as an extension of the weakened Hilbert's 16th problem to cylinder systems.
Lots of results have been derived for the study of limit cycles of system \eqref{h1},
for example see
\cite{AGP2007,GGM2016,GR2019,H1994,HJ1996,HW1996,SHZ2021, Y2021}.

Gasull, Geyer and Ma\~nosas \cite{GGM2016} considered the perturbed pendulum-like system
    \begin{equation}\label{h3}
     \dot x=y,\quad
       \dot y=-\sin(x)+\varepsilon \sum_{s=0}^{m}Q_{n,s}(x)y^{s},\,\,
    \end{equation}
    where the functions $Q_{n,0}(x)$, $\ldots$, $Q_{n,m}(x)$ are trigonometric polynomials of degree at most $n$.
Note that the unperturbed system of \eqref{h3} is
\begin{equation}\label{h3a}
\dot x = y,\quad \dot y=-\sin(x)
    \end{equation}
with the Hamiltonian
\[H(x,y)=\frac{1}{2}y^2+1-\cos (x)\]
defined on the cylinder $[-\pi,\pi]\times \mathbb{R}$.
    System \eqref{h3a} has an oscillatory region given by $\bigcup_{h\in(0,2)}\Gamma_h$ and
    a rotary region with two continuous families of periodic orbits $\bigcup_{h\in(2,\infty)}\Gamma_h$.
    The origin is a center surrounded by periodic orbits in the oscillatory region.
 Upper bounds were obtained in \cite{GGM2016} for the number of zeros
of the first order Melnikov function in both the oscillatory region and the rotary region of system \eqref{h3}, respectively.
    In addition, Yang \cite{Y2021} gave upper bounds of limit cycles
    for piecewise smooth perturbations of system \eqref{h3} with switching lines given by $x=0$ and $y=0$
     by studying the first order Melnikov function.
    Shi, Han and Zhang \cite{SHZ2021} derived the asymptotic expansions of the first order Melnikov function
    near a double homoclinic loop for system \eqref{h1}, and investigated the number of limit cycles near the double
    homoclinic loop $\Gamma_2$ for system \eqref{h3} as an application of their results.

More relevant results on pendulum systems can be found in \cite{BPS1977, CT2020, I1988,L1999,M2004,M1989,SC1986}.

    Inspired by \cite{GGM2016}, we consider the following cylinder near-Hamiltonian system
    \begin{equation}\label{b11}
       \dot x=y,\quad
       \dot y=-\sin(x)+\varepsilon \sum_{s=s_1}^{s_2}Q_{n,s}(x)y^{s},\,\,
    \end{equation}
    where $s_1$ and $s_2$ are two natural numbers with $s_1\le s_2$, and
    \begin{equation}\label{h4}
    Q_{n,s}(x)=\sum_{i=0}^n a_{i,s}\cos^i(x)+\sin(x)\sum_{i=0}^{n-1} \tilde a_{i,s}\cos^i(x)
    \end{equation}
    with $a_{i,s}$ and $\tilde a_{i,s}$ as free parameters.

    In this paper, for \eqref{b11} we shall study the maximum number of isolated zeros in $0<h\ll 1$ of the first order Melnikov
    function
     \begin{equation}\label{h5}
    M(h)=\oint_{\Gamma_h}\sum_{s=s_1}^{s_2}Q_{n,s}(x)y^{s}\mathrm{d}x,\quad h\in (0,2).
    \end{equation}
    We have the following theorem.

    \begin{thm}\label{thm1}
    Let \eqref{h4} and \eqref{h5} hold. Then the Melnikov function $M(h)$ in \eqref{h5} has the following form
    \begin{equation*}
    M(h)=\oint_{\Gamma_h}\sum_{k=r}^{r+m-1}Q_{n,2k+1}^{e}(x)y^{2k+1}\mathrm{d}x, \quad h\in(0,2),
    \end{equation*}
    where $r=[\frac{s_{1}}{2}]$, $m=[\frac{s_{2}-2r+1}{2}]$, $Q_{n,2k+1}^{e}(x)$ is the even part of $Q_{n,2k+1}(x)$.
    When $M(h) \not\equiv0$, there exists $0<\varepsilon_0\ll 1$ such that on the interval $h\in(0,\varepsilon_0)$
    $M(h)$ has at most $n+2m-2$ zeros (counting multiplicity) as $n>0$,
     or at most $m-1$ zeros (counting multiplicity) as $n=0$.
    This upper bound is sharp.
    \end{thm}

\begin{remark}
Note that the integer $m$ in Theorem \ref{thm1} represents the number of odd integers in $[s_1,s_2]$.
A conjecture was presented in \cite{GGM2016} that the sharp upper bound on the number of zeros  of $M(h)$ in \eqref{h5} on the interval $(0,2)$
is $n+2m-2$ for $n>0$ or $m-1$ for $n=0$. Our result in Theorem \ref{thm1} shows that this conjecture is true for $M(h)$
in a sufficiently small neighborhood of $h=0$.
\end{remark}

    Next, for system \eqref{h3a} we consider piecewise smooth perturbations given by
    \begin{equation}\label{b12}
    \left\{\begin{array}{l}
     \!\!\dot x=y,\\[1ex]
    \displaystyle \!\!\dot y=-\sin(x)+\varepsilon Q^{\pm}(x,y),\,\, {\pm}y>0,\,\,
     \end{array}\right.
    \end{equation}
    where
    \begin{equation}\label{b13}
    Q^{+}(x,y)=\sum_{s=s_1}^{s_2}Q_{n,s}^{+}(x)y^{s},\quad Q^{-}(x,y)=\sum_{s=s_1}^{s_3}Q_{n,s}^{-}(x)y^{s},
    \end{equation}
    with positive integers $s_1$, $s_2$ and $s_3$, and $Q_{n,s}^{\pm}(x)$ are  trigonometric polynomials of degree $n$.

     For system \eqref{b12} the associated first order Melnikov function $\widetilde M(h)$ on the oscillatory region is given by
    \begin{equation}\label{b14}
     \widetilde M(h)=\int_{\Gamma_h^+}Q^{+}(x,y)\mathrm{d}x+\int_{\Gamma_h^-}Q^{-}(x,y)\mathrm{d}x,\quad h\in (0,2),
    \end{equation}
    where $\Gamma_h^{\pm}=\Gamma_h\bigcap\{\pm y>0\}$. We have the next main result given in the following theorem.

    \begin{thm}\label{thm2}
    Let \eqref{b13} and \eqref{b14} hold with $s_i\ge 1$, $i=1,2,3$, and $\hat s=\max\{s_2,s_3\}$.
    When $\widetilde  M(h) \not\equiv0$, there exists $0<\varepsilon_0\ll 1$ such that on the interval $h\in(0,\varepsilon_0)$
    $\widetilde  M(h)$ has at most $\hat s-s_1$ zeros (counting multiplicity) as $n=0$, or at most $n$ zeros (counting multiplicity)  as $n>0$ and $\hat s=s_1$,
    or at most $2(n+\hat s- s_1)-1$ zeros (counting multiplicity) as $n>0$ and $\hat s>s_1$. This upper bound is sharp.
    \end{thm}

The paper is organized as follows. In Section \ref{sec2}, we present some preliminary results on
two types of Abelian integrals $I_{i,j}(h)$ and $J_{i,j}(h)$. In Section \ref{sec3},
we shall prove Theorems \ref{thm1} and \ref{thm2} by studying the asymptotic expansions of Melnikov functions $M(h)$
and $\widetilde M(h)$ near $h=0$, respectively.

\section{Preliminaries}\label{sec2}
In this section, we obtain some identities for two types of Abelian integrals in Lemmas \ref{lm2} and \ref{lm4}
to get generating functions for Melnikov functions in \eqref{h5} and \eqref{b14}, respectively.
Furthermore, we present two theorems which are useful for determining the maximum number
of isolated zeros (counting multiplicity) for $M(h)$ and $\widetilde M(h)$ in $0<h\ll 1$.

We begin with simplifying the Melnikov function $M(h)$ in \eqref{h5}.
Because every periodic orbit $\Gamma_h$, $h\in(0,2)$,
is symmetric with respect to the $x$-axis and $y$-axis,
 we have for any $k\ge0$
     \begin{equation*}
     \begin{split}
     &\oint_{\Gamma_h}Q_{n,2k}(x)y^{2k}\mathrm{d}x\equiv0, \\
     &\oint_{\Gamma_h}Q_{n,2k+1}(x)y^{2k+1}\mathrm{d}x=\oint_{\Gamma_h}Q^e_{n,2k+1}(x)y^{2k+1}\mathrm{d}x,
     \end{split}
     \end{equation*}
where $Q_{n,2k+1}^{e}(x)$ is the even part of $Q_{n,2k+1}(x)$.
Then $M(h)$ can be simplified into the form
     \begin{equation}\label{c1}
     M(h)=\oint_{\Gamma_h}\sum_{k=r}^{r+m-1}Q_{n,2k+1}^{e}(x)y^{2k+1}\mathrm{d}x \triangleq M_m(h), \quad h\in(0,2),
     \end{equation}
where $r=[\frac{s_{1}}{2}]$ and $m=[\frac{s_{2}-2r+1}{2}]$.

Note that by \eqref{h4} the function $Q_{n,2k+1}^{e}(x)$ can be written in the form
    \begin{equation}\label{c1a}
    Q_{n,2k+1}^e(x)=\sum_{i=0}^n \tilde c_{i,2k+1}(1-\cos(x))^i,
    \end{equation}
    where the coefficients $\tilde c_{i,2k+1}$ can be taken as independent parameters.
    Substituting \eqref{c1a} into \eqref{c1} yields
    \begin{equation}\label{c1b}
    M_m(h)=\sum_{j=r}^{r+m-1}\sum_{i=0}^n \tilde c_{i,2j+1} I_{i,2j+1}(h),
    \end{equation}
    where
    \begin{equation}\label{b22}
    I_{i,j}(h)=\oint_{\Gamma_h}(1-\cos (x))^i y^j \mathrm{d}x.
    \end{equation}

   To get the algebraic structure of $M_m(h)$, for Abelian integrals $I_{i,2j+1}(h)$ we have the following lemma.

\begin{lm}\label{lm2} 
    \rm{(I)} For any integers $i\ge0$ and $r\ge 0$ we have
    \begin{equation}\label{za}
    I_{i,2r+3}=\frac{2(2r+3)}{2i+1}I_{i+1,2r+1}
    + \frac{1}{2i+1}[(i+1)I_{i+1,2r+3}-(2r+3)I_{i+2,2r+1}].
    \end{equation}

    \rm{(II)} For any integers $i\ge 0$, $r\ge 0$ and $k\ge0$ we have
    \begin{equation}\label{zb}
    \begin{split}
    I_{i,2r+3}=&\,\frac{2(2r+3)}{2i+1}I_{i+1,2r+1}
    -(2r+3)\sum_{j=0}^{j\le k-1}\frac{(i+j)!(2i-1)!!}{i!(2i+2j+3)!!}I_{i+j+2,2r+1}\\
    &    +\frac{(i+k)!(2i-1)!!}{i!(2i+2k+1)!!}[(i+k+1)I_{i+k+1,2r+3}-(2r+3)I_{i+k+2,2r+1}].
    \end{split}
    \end{equation}

    \rm{(III)} For any integers $i\ge 0$ and $j\ge0$, $I_{i,2j+1}(h)$ can be expanded as
    \begin{equation}\label{z3a}
        I_{i,2j+1}(h)=h^{i+j+1}\sum_{k=0}^{\infty}\tilde b_k b_{i,j}^k h^k,\quad 0\le h\ll 1,
    \end{equation}
    where
    \begin{equation}\label{z3b}
    \tilde b_k = \frac{2^{3-k}\Gamma(k+\frac{1}{2})}{k!\Gamma(\frac{1}{2})},\quad
    b_{i,j}^k = \frac{2^{j-1}\Gamma(i+k+\frac{1}{2})\Gamma(j+\frac{3}{2})}{\Gamma(i+j+k+2)}.
    \end{equation}
    \end{lm}

    \begin{proof} (I)  For $I_{i+1,2r+3}(h)$ we have
    \begin{equation*}
    \begin{split}
    I_{i+1,2r+3}&=\oint_{\Gamma_h}(1-\cos (x))^{i+1} y^{2r+3}\mathrm{d}x\\
    &=I_{i,2r+3}-\oint_{\Gamma_h}\cos (x)(1-\cos (x))^i y^{2r+3}\mathrm{d}x.
    \end{split}\end{equation*}
    By integration by parts, we further get
    \begin{equation*}
    \begin{split}
    I_{i+1,2r+3}&=I_{i,2r+3}-\oint_{\Gamma_h}(1-\cos (x))^i y^{2r+3}\mathrm{d}\sin (x)\\
    &=I_{i,2r+3}+i\oint_{\Gamma_h}\sin^2 (x)(1-\cos (x))^{i-1} y^{2r+3}\mathrm{d}x\\
    &\qquad\qquad\quad\,\,+(2r+3)\oint_{\Gamma_h}\sin (x)(1-\cos (x))^i y^{2r+2}\mathrm{d}y.
    \end{split}\end{equation*}
    Noting that $\sin^2 (x)=-(1-\cos (x))^2+2(1-\cos (x))$ and $\mathrm{d}y=-\frac {\sin (x)}{y}\mathrm{d}x$ along $\Gamma_h$, we obtain
    \begin{equation*}
    \begin{split}
    I_{i+1,2r+3}=&-iI_{i+1,2r+3}+(2i+1)I_{i,2r+3}
    \\    & \qquad\qquad\qquad\,\,\,
    -(2r+3)\oint_{\Gamma_h}\sin^2 (x)(1-\cos (x))^i y^{2r+1}\mathrm{d}x\\
    =&-iI_{i+1,2r+3}+(2i+1)I_{i,2r+3}\\
    &\qquad\qquad\qquad\,\,\,+(2r+3)I_{i+2,2r+1}-2(2r+3)I_{i+1,2r+1},
    \end{split}\end{equation*}
    from which we can have \eqref{za}.

    (II) We shall prove \eqref{zb} by induction on $k$.
    When $k=0$, it is evident that \eqref{zb} holds by \eqref{za}.
    Suppose that \eqref{zb} holds for $k=m\ge0$, that is,
     \begin{equation}\label{zc}
     \begin{split}
    I_{i,2r+3}=&\,\frac{2(2r+3)}{2i+1}I_{i+1,2r+1}
    -(2r+3)\sum_{j=0}^{j\le m-1}c_{i,j}I_{i+j+2,2r+1}\\
    &\quad+e_{i,m}[(i+m+1)I_{i+m+1,2r+3}-(2r+3)I_{i+m+2,2r+1}],
    \end{split}
    \end{equation}
    where
    \begin{equation}\label{zb1}
    c_{i,j}=\frac{(i+j)!(2i-1)!!}{i!(2i+2j+3)!!},\quad
    e_{i,m}=\frac{(i+m)!(2i-1)!!}{i!(2i+2m+1)!!}.
    \end{equation}
    By \eqref{za} for $I_{i+m+1,2r+3}(h)$ we have
    \begin{equation*}
    \begin{split}
    I_{i+m+1,2r+3}=&\,\frac {2(2r+3)}{2i+2m+3}I_{i+m+2,2r+1}\\
    &\qquad+\frac {(i+m+2)I_{i+m+2,2r+3}-(2r+3)I_{i+m+3,2r+1}}{2i+2m+3}.
    \end{split}
    \end{equation*}
    Then we get
    \begin{equation}\label{zd}
    \begin{split}
    &(i+m+1)I_{i+m+1,2r+3}-(2r+3)I_{i+m+2,2r+1}=\tilde c_{i,m}I_{i+m+2,2r+1}\\
    &\qquad+\frac{i+m+1}{2i+2m+3}[(i+m+2)I_{i+m+2,2r+3}-(2r+3)I_{i+m+3,2r+1}],
    \end{split}\end{equation}
    where $\tilde c_{i,m}=
    -\frac{2r+3}{2i+2m+3}$.
    Then by \eqref{zb1} we can have $e_{i,m} \tilde c_{i,m}=-(2r+3)c_{i,m}$.
    Then substituting \eqref{zd} into \eqref{zc} yields \eqref{zb} with $k=m+1$.

(III)  For $I_{i,2j+1}(h)$ we have
    \begin{equation*}
    \begin{aligned}
    I_{i,2j+1}(h)
&=\oint_{\Gamma_h}(1-\cos (x))^i y^{2j+1}\mathrm{d}x
=4\int_{0}^{x_{+}(h)}(1-\cos (x))^i y_+^{2j+1}(x,h) \mathrm{d}x\\
&=4\int_{0}^{x_{+}(h)}2^{i}\sin^{2i} \left(\frac {x}{2}\right)
\left[2\left(h-2\sin^{2}\left(\frac {x}{2}\right)\right)\right]
^{\frac {2j+1}{2}}\mathrm{d}x,
    \end{aligned}\end{equation*}
where $y_+(x,h)=\sqrt {2h-2(1-\cos (x))}$ and $x_+(h)=\arccos(1-h)$.
Substituting $u=\sqrt{2}\sin(\frac{x}{2})$ into the above integral yields
    \begin{equation*}
    I_{i,2j+1}(h)
    =2^3\int_{0}^{\sqrt{h}}u^{2i}[2(h-u^{2})]^{\frac {2j+1}{2}} (2-u^2)^{-\frac{1}{2}}\mathrm{d}u.
    \end{equation*}
Let $u=\sqrt{h}\omega$. We further obtain
    \begin{equation*}
    I_{i,2j+1}(h)
    =2^{j+3}h^{i+j+1}\int_{0}^{1}\omega^{2i}(1-\omega^{2})^{\frac {2j+1}{2}}
\left(1-\frac{\omega^2}{2}h\right)^{-\frac{1}{2}}\mathrm{d}\omega.
    \end{equation*}
    Because $(1-\frac{\omega ^2}{2}h)^{-\frac{1}{2}}$ can expanded as
    \begin{equation*}
    \left(1-\frac{\omega ^2}{2}h\right)^{-\frac{1}{2}}=1+\frac{\omega ^2}{4}h+ \cdots +\frac{2^{-k}}{k!}\frac{\Gamma(k+\frac{1}{2})}{\Gamma(\frac{1}{2})}\omega^{2k}h^k+\cdots,\quad 0\le h\ll 1,
    \end{equation*}
    we further get
    \begin{equation}\label{ze}
    I_{i,2j+1}(h)=2^{j+3}h^{i+j+1}\sum_{k=0}^{\infty}\left(\frac{2^{-k}}{k!}\frac{\Gamma(k+\frac{1}{2})}{\Gamma(\frac{1}{2})}\int_{0}^{1}\omega^{2i+2k}(1-\omega^{2})^{\frac {2j+1}{2}}\mathrm{d}\omega \right) h^k.
    \end{equation}
    Note that by \eqref{z3b} we have
    \begin{equation*}
    2^j \int_{0}^{1}\omega^{2i+2k}(1-\omega^{2})^{\frac {2j+1}{2}}\mathrm{d}\omega=b^k_{i,j}.
    \end{equation*}
    Then \eqref{ze} can be simplified into the form \eqref{z3a} with \eqref{z3b} holding.
    \end{proof}

Let $\boldsymbol{\delta}\in\mathbb{R}^N$ be a vector consisting of all the free parameters in system \eqref{b11}.
Suppose that for any $\boldsymbol{\delta}\in\mathbb{R}^N$ the Melnikov function $M(h,\boldsymbol{\delta})$ of \eqref{b11} can be expanded in the form
\begin{equation}\label{t1}
M(h,\boldsymbol{\delta})=h^K(B_0(\boldsymbol{\delta})+B_1(\boldsymbol{\delta})h + \cdots +B_i(\boldsymbol{\delta})h^i +\cdots), \quad 0\le h\ll 1,
\end{equation}
where $K$ is a positive integer constant.
We can use the coefficients in \eqref{t1} to study the number of positive zeros of $M(h,\boldsymbol{\delta})$ near $h=0$.

\begin{thm}\label{thm3}
{\rm(I)} Suppose that there exist $k\ge 1$ and $\boldsymbol{\delta}_0\in \mathbb{R}^N$ such that
\begin{equation*}
B_i(\boldsymbol{\delta}_0)=0,\,\,i=0,1,\cdots,k-1,\,\, \mbox{and}\,\, B_{k}(\boldsymbol{\delta}_0)\neq0.
\end{equation*}
Then there exist $\varepsilon_0>0$ and a neighborhood $U\subset \mathbb{R}^N$ of $\boldsymbol{\delta}_0$
such that $M(h,\boldsymbol{\delta})$ has at most $k$ positive zeros in $0<h<\varepsilon_0$ for $\boldsymbol{\delta}\in U$.

{\rm(II)} Let each coefficient $B_i(\boldsymbol{\delta})$ in \eqref{t1} be linear in $\boldsymbol{\delta}$.
Suppose that there exists $k\ge 1$ such that
\begin{equation}\label{t1a}
\begin{split}
&\mbox{\rm rank} \left(\frac{\partial(B_0,B_1,\cdots,B_{k})}{\partial(\delta_1,\cdots,\delta_N)}\right)=k+1,\\
\mbox{and}\,\, &M(h,\boldsymbol{\delta})\equiv0\,\,\mbox{when}\,\, B_i=0,\,\, i=0,\cdots,k,
\end{split}
\end{equation}
where $\boldsymbol{\delta}=(\delta_1,\cdots,\delta_N)$.
Then for any compact set $D\subseteq\mathbb{R}^N$,
there exists $\varepsilon_0>0$
such that the non-vanishing $M(h,\boldsymbol{\delta})$ has at most $k$ positive zeros (counting multiplicity) in $0<h<\varepsilon_0$ for $\boldsymbol{\delta}\in D$.
Moreover, $M(h,\boldsymbol{\delta})$ can have $k$ simple positive zeros in an arbitrary neighborhood of $h=0$
for some $\boldsymbol{\delta}$.
\end{thm}

\begin{remark} If the conditions in the statement (I) of Theorem \ref{thm3} hold,
Theorem 2.4.1 in \cite{Han2013} shows that system \eqref{b11} has at most $k+K-1$ limit cycles
in a neighborhood of the origin for $\varepsilon$ sufficiently small and $\boldsymbol{\delta}$
 sufficiently close to $\boldsymbol{\delta}_0$.
Furthermore, when $K\geq 2$, the number of limit cycles bifurcating around the origin may be greater
than the number of positive zeros of $M(h,,\boldsymbol{\delta})$. An example is presented to illustrate this phenomenon in \cite{HLY2018}.
\end{remark}

Following the idea in the proof of Theorem 2.3.2 in \cite{Han2013}, we can similarly prove Theorem \ref{thm3}.
For the completeness of the paper, we present the proof below.

\begin{proof} (I)  We shall prove the statement (I) by contradiction.
Assume that for any $\varepsilon_0>0$ and an arbitrary neighborhood $U$ of $\boldsymbol{\delta}_0$,
$M(h,\boldsymbol{\delta})$ has $k+1$ positive zeros in $0<h<\varepsilon_0$ for some $\boldsymbol{\delta}\in U$.
Then there exist sequences $\boldsymbol{\delta}_n\rightarrow \boldsymbol{\delta}_0$ and $h_{nj}\rightarrow 0,\,\,j=1,2,\cdots ,k+1$ as $n\rightarrow \infty$
such that $\overline M(h_{nj},\boldsymbol{\delta}_n)=0$, where $\overline M(h,\boldsymbol{\delta})=M(h,\boldsymbol{\delta})/h^K$.
By Rolle's theorem $\frac {\partial \overline M}{\partial h}(h,\boldsymbol{\delta}_n)$ has $k$ positive zeros near $h=0$.
Then applying Rolle's theorem repeatedly, we can see that $\frac {\partial^k \overline M}{\partial h^k}(h,\boldsymbol{\delta}_n)$ has a positive zero $h_n^*$ and $h_n^*\rightarrow 0$ as $n\rightarrow \infty$. Then we have $\frac {\partial^k \overline M}{\partial h^k}(0,\boldsymbol{\delta}_0)=0$. From \eqref{t1} we get that
\begin{equation*}
    \overline M(h,\boldsymbol{\delta})
    =B_0(\boldsymbol{\delta})+B_1(\boldsymbol{\delta})h+\cdots+B_k(\boldsymbol{\delta})h^k+\cdots,
    \end{equation*}
which implies $\frac {\partial^k \overline M}{\partial h^k}(0,\boldsymbol{\delta}_0)=k!B_k(\boldsymbol{\delta}_0)\neq0$. This is a contradiction.
Therefore, there exist $\varepsilon_0>0$ and a neighborhood $U$ of $\boldsymbol{\delta}_0$
such that $M(h,\boldsymbol{\delta})$ has at most $k$ positive zeros in $0<h<\varepsilon_0$ for $\boldsymbol{\delta}\in U$.

(II) Without loss of generality, we suppose
\begin{equation}\label{t2}
  \mbox{\rm rank} \left(\frac{\partial(B_0,B_1,\cdots,B_{k})}{\partial(\delta_1,\delta_2,\cdots,\delta_{k+1})}\right)=k+1.
     \end{equation}
Because  in \eqref{t1} all coefficients $B_i(\boldsymbol{\delta})$, $i\geq0$ are liner in
 $\boldsymbol{\delta}$, by \eqref{t2} we can get linear combinations
 $\delta_i=\tilde{\delta_i}(B_0,B_1,\cdots,B_k,\delta_{k+2},\cdots,\delta_N),\,\, i=1,2,\cdots,k+1$ from the equations $B_j(\boldsymbol{\delta})=B_j,\,\,j=0,1,\cdots,k$.
Substituting $\delta_i=\tilde{\delta_i},\,\,i=1,2,\cdots,k+1$
 into $B_j(\boldsymbol{\delta}),\,\,j\geq k+1$ yields linear combinations $B_j(\boldsymbol{\delta})=\widetilde B_j(B_0,B_1,\cdots,B_k,\delta_{k+2},\cdots,\delta_N)$.
Then by \eqref{t1a} for any $j\geq k+1$,
\[\widetilde B_j(0,0,\cdots,0,\delta_{k+2},\cdots,\delta_N)\equiv 0,\]
which implies $\widetilde B_j$ does not depends on $\delta_{k+2},\cdots,\delta_N$.
Then $B_j(\boldsymbol{\delta})$ can be written as $B_j(\boldsymbol{\delta})=\widetilde B_j(B_0,B_1,\cdots,B_k)$. Because $M(h,\boldsymbol{\delta})$ is analytic in $h$, from \eqref{t1} we have $M(h,\boldsymbol{\delta})=h^K\overline M(h)$, where
\begin{equation}\label{t3}
  \overline M(h)=\sum_{j=0}^{k}B_jh^j(1+P_j(h)),
     \end{equation}
and $P_j(h)=O(h)\in {C^\omega}$. Then to prove the statement (II) we only need to prove $\overline M_1(h)= \overline M(h)/(1+P_0(h))$ has at most $k$ positive zeros in $0<h<\varepsilon_0$.

Note that by \eqref{t3} we have
\begin{equation}\label{t4}
  \overline M_1(h)=B_0+\sum_{j=1}^{k}B_jh^j(1+\widetilde P_j(h)),
     \end{equation}
where $\widetilde P_j(h)=\frac {1+P_j(h)}{1+P_0(h)}-1=\frac {P_j(h)-P_0(h)}{1+P_0(h)}=O(h)$.
Note that for $k=0$, the non-vanishing constant function $\overline M_1(h)=B_0$ has none zeros in $h$. For $k\geq1$ by \eqref{t4} we have
\begin{equation}\label{t5}
 \frac {\partial \overline M_1}{\partial h}=\sum_{j=1}^{k}jB_jh^{j-1}(1+P_{2j}(h))
 =\overline M_2(h)(1+P_{21}(h)),
     \end{equation}
where $P_{2j}(h)=\widetilde P_{j}(h)+\frac {1}{j}h\widetilde P_{j}^{'}(h)=O(h)$.
Then it suffices to show $\overline M_2(h)$ has at most $k-1$ positive zeros,
where by \eqref{t5} we get
\begin{equation*}
 \overline M_2(h)=\frac {\overline M_{1h}(h)}{1+P_{21}(h)}=B_1+\sum_{j=2}^{k}jB_j h^{j-1}(1+\widetilde P_{2j}(h)).
     \end{equation*}
Because $\overline M_1(h)$ and $\overline M_2(h)$ have the same form, we can use induction on $k$ to prove the conclusion for $k\geq 2$.

Because functions $f_j(h)=h^j(1+P_j(h))$ are independent, there exist proper values for $B_0,B_1,\cdots,B_k$ such that $\overline M_1(h)=\sum_{j=0}^{k}B_jf_j(h)$ has $k$ positive zeros near $h=0$, which implies that $M(h)$ can have $k$ positive zeros in $0<h<\varepsilon_0$ for some $\boldsymbol{\delta}$.
\end{proof}

     Next we consider the piecewise smooth perturbations given in \eqref{b12}.
     Substituting \eqref{b13} into \eqref{b14} yields
     \begin{equation}\label{t5a}
     \widetilde M(h)= \sum_{s=s_1}^{s_2} \int_{\Gamma^+_h} Q_{n,s}^{+}(x)y^{s}\mathrm{d}x
         + \sum_{s=s_1}^{s_3}\int_{\Gamma^-_h}Q_{n,s}^{-}(x)y^{s}\mathrm{d}x.
     \end{equation}
     For any integer $s\ge0$ we have
     \begin{equation*}
     \begin{split}
     & \int_{\Gamma_h^-}Q_{n,s}^{-}(x)y^{s}\mathrm{d}x=(-1)^{s+1}\int_{\Gamma_h^+}Q_{n,s}^{-}(x)y^{s}\mathrm{d}x\,\\
     \mbox{and} \,\, &  \int_{\Gamma_h^\pm}f(x)y^{s}\mathrm{d}x\equiv 0\,\,\mbox{for odd continuous functions} \,\, f(x).
     \end{split}
     \end{equation*}
     Then $\widetilde M(h)$ in \eqref{t5a} can be changed into the form
     \begin{equation*}
     \widetilde M(h)=\sum_{s=s_1}^{\hat s}\int_{\Gamma_h^+}\widetilde{Q}_{n,s}^{e}(x)y^{s}\mathrm{d}x,
     \end{equation*}
     where $\hat s=\max\{s_2,s_3\}$, and $\widetilde{Q}_{n,s}^{e}(x)$ is the even part of ${Q}^+_{n,s}(x)+(-1)^{s+1}{Q}^-_{n,s}(x)$
     with $Q_{n,s}^+=0$ for $s_2<s\le \hat s$, and $Q_{n,s}^-=0$ for $s_3<s\le \hat s$.
     Note that $\widetilde Q_{n,s}^e(x)$ can be written in the form
     \begin{equation*}
     \widetilde Q_{n,s}^e(x)= \sum_{i=0}^n b_{i,s}(1-\cos(x))^i,
     \end{equation*}
     where all the coefficients $b_{i,s}$'s are free parameters.
    Then $\widetilde M(h)$ can be simplified into
      \begin{equation}\label{t6}
     \begin{aligned}
     \widetilde M(h)&=\sum_{s=s_1}^{\hat s}\int_{\Gamma_h^+}\widetilde{Q}_{n,s}^{e}(x)y^{s}\mathrm{d}x
     = \sum_{s=s_1}^{\hat s}\sum_{i=0}^n b_{i,s}J_{i,s}(h),
     \end{aligned}\end{equation}
     where
     \begin{equation}\label{zef}
     J_{i,s}(h)=\int_{\Gamma_h^+}(1-\cos(x))^iy^s\mathrm{d}x.
     \end{equation}
     Note that by \eqref{b22} and \eqref{zef} $J_{i,2r+1}(h)=\frac{1}{2}I_{i,2r+1}(h)$ for any $r\ge0$. For $J_{i,2r}(h)$ we have the following lemma.

\begin{lm}\label{lm4}
    \rm{(I)} For any integers $i\ge0$ and $r\ge 1$ we have
 \begin{equation}\label{zf}
     J_{i,2r+2}=\frac{2(2r+2)}{2i+1} J_{i+1,2r}
    + \frac{1}{2i+1}[(i+1) J_{i+1,2r+2}-(2r+2) J_{i+2,2r}].
    \end{equation}
    \rm{(II)} For any integers $i\ge 0$ and $k\ge0$ we have
    \begin{equation}\label{zg}
    \begin{split}
     J_{i,2r+2}=&\,\frac{2(2r+2)}{2i+1} J_{i+1,2r}
    -(2r+2)\sum_{j=0}^{j\le k-1}\frac{(i+j)!(2i-1)!!}{i!(2i+2j+3)!!} J_{i+j+2,2r}\\
    &    +\frac{(i+k)!(2i-1)!!}{i!(2i+2k+1)!!}[(i+k+1) J_{i+k+1,2r+2}-(2r+2) J_{i+k+2,2r}].
    \end{split}
    \end{equation}

    \rm{(III)} For any integers $i\ge 0$ and $j\ge1$, ${J}_{i,2j}(h)$ can be expanded as
    \begin{equation}\label{zh1}
        {J}_{i,2j}(h)=h^{i+j+\frac{1}{2}}\sum_{k=0}^{\infty} \tilde b_k c_{i,j}^k h^k,\quad 0\le h\ll 1,
    \end{equation}
    where $\tilde b_k$ is given in \eqref{z3b}, and
    \begin{equation}\label{zh2}
    c_{i,j}^k = \frac{2^{j-\frac{5}{2}}\Gamma(i+k+\frac{1}{2})\Gamma(j+1)}{\Gamma(i+j+k+\frac{3}{2})}.
    \end{equation}
\end{lm}

\begin{proof} (I)  The identity \eqref{zf} can be similarly proved as in the proof of \eqref{za} in Lemma \ref{lm2}.
For $ J_{i+1,2r+2}(h)$ we have
    \begin{equation}\label{zi}
    \begin{split}
     J_{i+1,2r+2}(h)&=\int_{\Gamma_h^+}(1-\cos (x))^{i+1} y^{2r+2}\mathrm{d}x\\
    &= J_{i,2r+2}(h)-\int_{\Gamma_h^+}\cos (x)(1-\cos (x))^i y^{2r+2}\mathrm{d}x.
    \end{split}\end{equation}
    By using the method of integration by parts, we get
    \begin{equation*}
    \begin{split}
    &\int_{\Gamma_h^+}\cos (x)(1-\cos (x))^i y^{2r+2}\mathrm{d}x\\
    =&\, (1-\cos (x))^i y^{2r+2}\sin (x)\bigg\arrowvert_{\Gamma_h^+}
    -\int_{\Gamma_h^+}\sin (x)\mathrm{d}\left[(1-\cos (x))^i y^{2r+2}\right]\\
    =&\, -i\int_{\Gamma_h^+}\sin^2 (x)(1-\cos (x))^{i-1} y^{2r+2}\mathrm{d}x\\
    &\qquad-(2r+2)\int_{\Gamma_h^+}\sin (x)(1-\cos (x))^i y^{2r+1}\mathrm{d}y.
    \end{split}\end{equation*}
    Noting that $\mathrm{d}y=-\frac {\sin (x)}{y}\mathrm{d}x$ along $\Gamma_h^+$ and $\sin^2 (x)=-(1-\cos (x))^2+2(1-\cos (x))$, we obtain
    \begin{equation}\label{zj}
    \begin{split}
     &\int_{\Gamma_h^+}\cos (x)(1-\cos (x))^i y^{2r+2}\mathrm{d}x\\
    =&\,iJ_{i+1,2r+2}-2iJ_{i,2r+2}-(2r+2)(J_{i+2,2r}-2J_{i+1,2r}).
    \end{split}\end{equation}
    Then from \eqref{zi} and \eqref{zj} we can have \eqref{zf}.

(II) As in the proof for \eqref{zb}, we can use the method of induction on $k$ to prove \eqref{zg}.
The details of the proof are omitted here.

(III) We can obtain \eqref{zh1} for $J_{i,2j}(h)$ by the same way used in the proof for \eqref{z3a}. 
For $J_{i,2j}(h)$ we have
    \begin{equation*}
    \begin{aligned}
    J_{i,2j}(h)
&=\int_{\Gamma_h^+}(1-\cos (x))^i y^{2j}\mathrm{d}x
=2\int_{0}^{x_{+}(h)}(1-\cos (x))^i y_+^{2j}(x) \mathrm{d}x\\
&=2\int_{0}^{x_{+}(h)}2^{i}\sin^{2i} \left(\frac {x}{2}\right)
\left[2\left(h-2\sin^{2}\left(\frac {x}{2}\right)\right)\right]
^{j}\mathrm{d}x,
    \end{aligned}\end{equation*}
where $x_+(h)=\arccos(1-h)$.
By the variable transformation $\sqrt{h}\omega=\sqrt{2}\sin(\frac{x}{2})$ we further obtain
    \begin{equation}\label{ze1}
    \begin{aligned}
    J_{i,2j}(h)&=2^{j+\frac {3}{2}}h^{i+j+\frac {1}{2}}\int_{0}^{1}\omega^{2i}(1-\omega^{2})^{j}
\left(1-\frac{\omega^2}{2}h\right)^{-\frac{1}{2}}\mathrm{d}\omega\\
    &=2^{j-\frac {3}{2}}h^{i+j+\frac {1}{2}}
    \sum_{k=0}^{\infty}\tilde b_k\int_{0}^{1}\omega^{2i+2k}(1-\omega^{2})^{j}\mathrm{d}\omega  h^k.
    \end{aligned}\end{equation}
    Note that for $c^k_{i,j}$ in \eqref{zh2} we have
    \begin{equation*}
    2^{j-\frac {3}{2}} \int_{0}^{1}\omega^{2i+2k}(1-\omega^{2})^{j}\mathrm{d}\omega=c^k_{i,j}.
    \end{equation*}
    Then \eqref{ze1} can be simplified into the form \eqref{zh1}. The proof is completed.
\end{proof}

Let $\boldsymbol{\delta}\in\mathbb{R}^N$ be a vector consisting of all the parameters in system \eqref{b12}.
By \cite{LiuHan2010} for any $\boldsymbol{\delta}\in\mathbb{R}^N$ the Melnikov function $\widetilde M(h,\boldsymbol{\delta})$ of \eqref{b12} can be expanded in the form
\begin{equation}\label{zk}
\widetilde M(h,\boldsymbol{\delta})=h^K(B_0(\boldsymbol{\delta})+B_1(\boldsymbol{\delta})h^{\frac{1}{2}} \cdots +B_i(\boldsymbol{\delta})h^{\frac{i}{2}}+\cdots), \quad 0\le h\ll 1,
\end{equation}
where $2K$ is a positive integer.
We can use the coefficients in \eqref{zk} to study the number of positive zeros of $\widetilde M(h,\boldsymbol{\delta})$ near $h=0$.

\begin{thm}\label{thm4}
Let each coefficient $B_i(\boldsymbol{\delta})$ in \eqref{zk} be linear in $\boldsymbol{\delta}$.
Suppose that there exist natural numbers $k_1$ and $k_2$ such that for $i>k_1$ and $j>k_2$
\begin{equation}\label{zl}
\begin{split}
&B_{2i}=O(|B_0,B_2,\cdots,B_{2k_1}|),\,\quad B_{2j-1}=O(|B_1,B_3,\cdots,B_{2k_2-1}|),\\
&\mbox{\rm rank} \left(\frac{\partial(B_0,B_2,\cdots,B_{2k_1},B_1,B_3,\cdots,B_{2k_2-1})}{\partial(\delta_1,\cdots,\delta_N)}\right)=k_1+k_2+1,\,\,
\end{split}
\end{equation}
where $\boldsymbol{\delta}=(\delta_1,\cdots,\delta_N)$.
Then for any compact set $D\subseteq\mathbb{R}^N$,
when $\widetilde M(h,\boldsymbol{\delta})\not\equiv0$,
there exists $\varepsilon_0>0$
such that the non-vanishing $\widetilde M(h,\boldsymbol{\delta})$ has at most $k_1+k_2$ positive zeros in $0<h<\varepsilon_0$ for $\boldsymbol{\delta}\in D$.
Moreover, $\widetilde M(h, \boldsymbol{\delta})$ can have $k_1+k_2$ positive simple zeros in an arbitrary neighborhood of $h=0$
for some $\boldsymbol{\delta}$.
\end{thm}

\begin{proof}
Here we only give the proof for the case $k_2\ge k_1$. The other case $k_1>k_2$ can be similarly proved.

Let $\rho=\sqrt{h}$. As in the proof of Theorem \ref{thm3} $\widetilde M(\rho^2,\boldsymbol{\delta})$ can be written as
$\widetilde M(\rho^2,\boldsymbol{\delta})=\rho^{2K}\overline M(\rho)$, where
\begin{equation}\label{zl1}
\overline M(\rho)=\sum_{i=0}^{k_1}B_{2i}\rho^{2i}(1+P_{2i}(\rho^2))+\sum_{i=1}^{k_2} B_{2i-1}\rho^{2i-1}(1+P_{2i-1}(\rho^2)),
\end{equation}
where $P_j(\rho^2)=O(\rho^2)\in C^{\omega}$.
We shall use induction on $k_1$ to prove that $\overline M(\rho)$ has at most $k_1+k_2$ zeros in $0<\rho<\sqrt{\varepsilon_0}$.

When $k_1=0$, for any $k_2\ge 0$ by \eqref{zl1} we have
\begin{equation*}
\overline M(\rho)=B_0(1+P_{0}(\rho^2))+\sum_{i=1}^{k_2} B_{2i-1}\rho^{2i-1}(1+P_{2i-1}(\rho^2)),
\end{equation*}
which yields
\begin{equation}\label{zl2}
\frac{\overline M(\rho)}{1+P_0(\rho)}=B_0+\sum_{i=1}^{k_2} B_{2i-1}\rho^{2i-1}(1+\widetilde P_{2i-1}(\rho^2))\triangleq \widetilde M_0(\rho),
\end{equation}
where $\widetilde P_{2i-1}(\rho^2)=\frac{P_{2i-1}(\rho^2)-P_0(\rho^2)}{1+P_0(\rho^2)}=O(\rho^2)$.
Clearly, for $k_2=0$ $\widetilde M_0(\rho)$ has none positive zeros near $\rho=0$ when $\widetilde M(h)\not\equiv0$.
For $k_2\ge1$ we have
\begin{equation*}
\widetilde M'_0(\rho)= \sum_{i=1}^{k_2} (2i-1)B_{2i-1}\rho^{2i-2}(1+\overline P_{2i-1}(\rho^2)),
\end{equation*}
where $\overline P_{2i+1}(\rho^2)=\widetilde P_{2i-1}(\rho^2)+\frac{2}{2i-1}\rho^2\widetilde P'_{2i-1}(\rho^2)=O(\rho^2)$.
By the proof of Theorem \eqref{thm3} we get that $\widetilde M'_0(\rho)$ has at most $k_2-1$ positive zeros in $\rho^2$ near $\rho=0$,
which implies by \eqref{zl2} that $\widetilde M(\rho)$ has at most $k_2$ positive zeros in $0<\rho<\sqrt{\varepsilon_0}$.

Suppose that for $k_1=k\ge 0$ and $k_2\ge k$ the function $\overline M(\rho)$ has at most $k+k_2$ positive zeros.
Let $\widetilde M_{k+1}(\rho)=\frac{\overline M(\rho)}{1+P_0(\rho)}$.
Then for $k_1=k+1$ and $k_2\ge k$, by \eqref{zl1} we have
\begin{equation*}\label{zl3}
\widetilde M_{k+1}(\rho)=B_0+\sum_{i=1}^{k+1}B_{2i}\rho^{2i}(1+\widetilde P_{2i}(\rho^2))
+\sum_{i=1}^{k_2} B_{2i-1}\rho^{2i-1}(1+\widetilde P_{2i+1}(\rho^2)),
\end{equation*}
where $\widetilde P_{2i}(\rho^2)=\frac{P_{2i}(\rho^2)-P_0(\rho^2)}{1+P_0(\rho^2)}=O(\rho^2)$.
Let $\widehat M_{k+1}(\rho)= \frac{\widetilde M'_{k+1}(\rho)}{1+\overline P_1(\rho^2)}$.
Similarly, we have
\begin{equation*}
 \widehat M_{k+1}(\rho)=B_1+\sum_{i=1}^{k_2-1}(2i+1)B_{2i+1}\rho^{2i}(1+O(\rho^2))
+\sum_{i=1}^{k+1}2i B_{2i}\rho^{2i-1}(1+O(\rho^2)),
\end{equation*}
which yields
\begin{equation*}
\begin{split}
 \widehat M'_{k+1}(\rho)=& \sum_{i=0}^{k}(2i+2)(2i+1) B_{2i+2}\rho^{2i}(1+O(\rho^2))\\
&\qquad+ \sum_{i=1}^{k_2-1}2i(2i+1)B_{2i+1}\rho^{2i-1}(1+O(\rho^2)).
\end{split}
\end{equation*}
By the assumption of $\overline M(\rho)$ for $k_1=k$, we get that $\widehat M'_{k+1}(\rho)$ has at most
$k+k_2-1$ positive zeros near $\rho=0$.
Therefore, for $k_1=k+1$ and $k_2\ge k_1$ $\overline M(\rho)$ has at most $k+k_2+1$ positive zeros.
\end{proof}

\section{Proof of the Main Results}\label{sec3}
In this section, we shall present the proof for Theorems \ref{thm1} and \ref{thm2}, respectively.
To prove Theorem \ref{thm1}, we rewrite $M_m(h)$ in \eqref{c1b} as a linear combination of some Abelian integrals $I_{i,2j+1}$ in the next lemma.

    \begin{lm}\label{lm6}
    For any integers $m\ge 2$, $n\ge1$ and $r\ge0$,  the Melnikov function $M_m(h)$ in \eqref{c1b} can be simplified into the following form
    \begin{equation}\label{b31m}
    M_m(h)=\sum_{i=0}^{n+m-1}A_{i}\,I_{i,2r+1}+\sum_{i=1}^{m-1}A_{n+m-1+i}\,L_{n-1+i,r+m-i}(h),
    \end{equation}
    where the coefficients $A_{i}$, $i=0,\cdots, n+2(m-1)$ can be taken as free parameters, and
    $L_{i,j+1}(h)=(i+1)I_{i+1,2j+3}-(2j+3)I_{i+2,2j+1}$.
    \end{lm}

    \begin{proof}
    We shall prove \eqref{b31m} by induction on $m$. When $m=2$, for any $r\ge0$ by \eqref{c1b} we have
    \begin{equation}\label{c2}
    \begin{aligned}
    M_2(h)&
    =\sum_{j=r}^{r+1}\sum_{i=0}^{n} \tilde c_{i,2j+1}I_{i,2j+1}=\sum_{i=0}^{n} \tilde c_{i,2r+1}I_{i,2r+1}
    +\sum_{i=0}^{n}\tilde c_{i,2r+3}I_{i,2r+3}.
    \end{aligned}\end{equation}
    For any $0\leq i \leq n-2$, by \eqref{zb} we have
    \begin{equation}\label{c3}
    \begin{aligned}
    I_{i,2r+3}&=\frac {2(2r+3)}{2i+1}I_{i+1,2r+1}-(2r+3)\sum_{j=0}^{n-3-i}c_{i,j}I_{i+j+2,2r+1}\\
    &+\frac {(2i-1)!!(n-2)!}{(2n-3)!!i!}[(n-1)I_{n-1,2r+3}-(2r+3)I_{n,2r+1}],
    \end{aligned}\end{equation}
    where $c_{i,j}$ is given in \eqref{zb1}. Then substituting \eqref{c3} into \eqref{c2} yields
    \begin{equation}\label{c3b}
    M_2(h)=\sum_{i=0}^{n}\tilde c_{i} I_{i,2r+1}+\tilde c_{n+1} I_{n-1,2r+3}+\tilde c_{n+2} I_{n,2r+3},
    \end{equation}
    where the coefficients $\tilde c_{0},\tilde c_{1},\cdots,\tilde c_{n+2}$ are independent.
    For any $n\ge 1$, by \eqref{zb} with $i=n-1$ and $k=1$ we have
    \begin{equation}\label{c3a}
    I_{n-1,2r+3}=\frac {2(2r+3)}{2n-1}I_{n,2r+1}-\frac {2r+3}{4n^2-1}I_{n+1,2r+1}+\frac {n}{4n^2-1}L_{n,r+1}(h).
    \end{equation}
    Then using \eqref{c3a} and \eqref{za}$|_{i=n}$ to remove $I_{n-1,2r+3}$ and $I_{n,2r+3}$ in \eqref{c3b},
     we can have \eqref{b31m} for $m=2$ with
    \begin{equation}\label{c4}
    \begin{aligned}
    &A_{i}=\tilde c_{i},\quad 0\leq i \leq n-1,\\
    & A_{n}=\tilde c_{n}+\frac {2(2r+3)}{2n-1}\tilde c_{n+1},\\
    &A_{n+1}=-\frac {2r+3}{4n^2-1}\tilde c_{n+1}+\frac {2(2r+3)}{2n+1}\tilde c_{n+2},\\
    & A_{n+2}=\frac {n}{4n^2-1}\tilde c_{n+1}+\frac {1}{2n+1}\tilde c_{n+2}.
    \end{aligned}\end{equation}
    By \eqref{c4} we can get that $A_0,A_1,\cdots,A_{n+2}$ are linearly independent.

    Suppose \eqref{b31m} holds for $m=k\geq 2$. Then for any integers $n\geq 1$ and $r\geq 0$, we have
    \begin{equation}\label{c5}
    \begin{aligned}
    M_k(h)=\sum_{i=0}^{n+k-1}A_iI_{i,2r+1}+\sum_{i=1}^{k-1}A_{n+k-1+i} L_{n-1+i,r+k-i}(h),
    \end{aligned}\end{equation}
    where the coefficients $A_i$, $i=0,\cdots, n+2(k-1)$ can be taken as free parameters.
    When $m=k+1$, by \eqref{c1b} we have
    \begin{equation}\label{c6}
    M_{k+1}(h)=\sum_{i=0}^{n}\tilde c_{i,2r+1}I_{i,2r+1}+ M^{\ast}_k(h),\,\,\,\,
     M^{\ast}_k(h)=\sum_{j=r+1}^{r+k}\sum_{i=0}^n \tilde c_{i,2j+1}I_{i,2j+1}.
     \end{equation}
    Then by \eqref{c5} $M^{\ast}_k(h)$ can be rewritten as
    \begin{equation}\label{c7}
    \begin{split}
    &M^{\ast}_k(h)=\sum_{i=0}^{n+k-1}\tilde A_i I_{i,2r+3}+\tilde L_{k-1}(h),\\
    &\tilde L_{k-1}(h)=\sum_{i=1}^{k-1}\tilde A_{n+k-1+i}\, L_{n-1+i,r+1+k-i}(h).
    \end{split}
    \end{equation}
    Note that by \eqref{c3}, for any $0\leq i\leq n-2$ the integral $I_{i,2r+3}$ can be expressed
    as a linear combination of $I_{0,2r+1},I_{1,2r+1},\cdots,I_{n,2r+1}$ and $I_{n-1,2r+3}$.
    Then by \eqref{c3} and \eqref{c7}, $M_{k+1}(h)$ in \eqref{c6} can be further simplified into
    \begin{equation}\label{c8}
    \begin{aligned}
    M_{k+1}(h)=\sum_{i=0}^{n}\tilde c_{i} I_{i,2r+1}+\tilde c_{n+1} I_{n-1,2r+3} + \sum_{i=n}^{n+k-1}\tilde A_i I_{i,2r+3}+\tilde L_{k-1}(h),
    \end{aligned}\end{equation}
    where the coefficients $\tilde c_{i}$, $i=0,1,\cdots,n+1$ and $\tilde A_i$, $i=n,\cdots,n+2(k-1)$ are linearly
    independent. By \eqref{za}$|_{i=n-1}$, we have
    \begin{equation}\label{c9}
    \begin{aligned}
    I_{n-1,2r+3}=\frac {2(2r+3)}{2n-1}I_{n,2r+1}+\frac {1}{2n-1}[nI_{n,2r+3}-(2r+3)I_{n+1,2r+1}].
    \end{aligned}\end{equation}
    Substituting \eqref{c9} into \eqref{c8} yields
    \begin{equation}\label{c10}
    \begin{split}
    M_{k+1}(h)=&\sum_{i=0}^{n}B_i I_{i,2r+1}+\tilde B_{n+1} I_{n+1,2r+1}\\
    &\qquad+\tilde B_{n+2} I_{n,2r+3}+\sum_{i=n+1}^{n+k-1}\tilde A_i I_{i,2r+3}+\tilde L_{k-1}(h),
    \end{split}\end{equation}
    where
    \begin{equation*}
    \begin{aligned}
    &B_i=\tilde c_{i},\quad 0\leq i \leq n-1,\quad B_n=\tilde c_{n}+\frac {2(2r+3)}{2n-1}\tilde c_{n+1},\\
    &\tilde B_{n+1}=-\frac {2r+3}{2n-1}\tilde c_{n+1},\quad \tilde B_{n+2}=\frac {n}{2n-1}\tilde c_{n+1}+\tilde A_n,
    \end{aligned}\end{equation*}
    from which we can see that the coefficients in \eqref{c10} are also linearly independent.
    Similarly, substituting \eqref{za}$|_{i=n}$ into \eqref{c10}, we obtain
    \begin{equation}\label{c11}
    \begin{split}
    M_{k+1}(h)=&\sum_{i=0}^{n+1}B_{i+1} I_{i,2r+1}+\tilde B_{n+2} I_{n+2,2r+1}\\
    &\qquad+\tilde B_{n+3} I_{n+1,2r+3}+\sum_{i=n+2}^{n+k-1}\tilde A_i I_{i,2r+3}+\tilde L_{k-1}(h)
    \end{split}
    \end{equation}
    with independent coefficients. We can repeat this process to remove $I_{n+1,2r+3}$, $\cdots$, $I_{n+k-2,2r+3}$ from \eqref{c11} and get
    \begin{equation}\label{c12}
    \begin{aligned}
    M_{k+1}(h)=\!\!\sum_{i=0}^{n+k-1}B_i I_{i,2r+1}+\tilde B_{n+k} I_{n+k,2r+1}+\tilde B_{n+k+1} I_{n+k-1,2r+3}+\tilde L_{k-1}(h),
    \end{aligned}\end{equation}
    where the coefficients are still independent. Then using \eqref{za}$|_{i=n+k-1}$ to remove $I_{n+k-1,2r+3}$ from \eqref{c12}, we derive
    \begin{equation*}
    \begin{split}
    M_{k+1}(h)=&\sum_{i=0}^{n+k}B_i I_{i,2r+1}+\tilde L_{k-1}(h)+\tilde A_{n+2k-1}L_{n+k-1,r+1}(h)\\
    =&\sum_{i=0}^{n+k}B_i I_{i,2r+1}+\sum_{i=1}^{k}\tilde A_{n+k-1+i}\, L_{n-1+i,r+1+k-i}(h),
    \end{split}\end{equation*}
    where
    \begin{equation*}
    \begin{aligned}
    B_{n+k}=\tilde B_{n+k} +\frac {2(2r+3)}{2n+2k-1}\tilde B_{n+k+1},\qquad \tilde A_{n+2k-1}=\frac {1}{2n+2k-1}\tilde B_{n+k+1}.
    \end{aligned}\end{equation*}
    Note that the coefficients $B_{i}$, $i=0,1,\cdots,n+k$, and $\tilde A_{i}$, $i=n+k,\cdots,n+2k-1$, are independent.
    Then \eqref{b31m} is proved for $m=k+1$. The proof is completed.
    \end{proof}

Next, we present the proof of Theorem \ref{thm1} as below.

\begin{proof} [Proof of Theorem \ref{thm1}]
By Theorem \ref{thm3} we need to study the asymptotic expansion of $M_m(h)$ in \eqref{c1b} near $h=0$.
Here we only present the proof for the case of $n\ge 1$.  Using (III) of Lemma \ref{lm2}
the conclusion for $n=0$ can be similarly proved.

Let $n\ge 1$. We firstly consider the case of $m=1$. From \eqref{c1b} we have
     \begin{equation*}
     \begin{aligned}
     M_1(h)=\oint_{\Gamma_h}\sum_{i=0}^{n}\tilde c_{i,2r+1}(1-\cos(x))^iy^{2r+1}\mathrm{d}x=\sum_{i=0}^{n}\tilde c_{i,2r+1}I_{i,2r+1},
    \end{aligned}\end{equation*}
where the coefficients $\tilde c_{i,2r+1}$'s can be taken as independent parameters.
Then by (III) of Lemma \ref{lm2}, we get
\begin{equation}\label{d0}
        I_{i,2r+1}(h)=h^{i+r+1}\sum_{k=0}^{+\infty}\tilde b_k b_{i,r}^k h^k,\quad 0\le h\ll 1,
    \end{equation}
    where $\tilde b_k$ and $b_{i,r}^k$ are given in \eqref{z3b}.
Then $M_1(h)$ has the following asymptotic expansion
\begin{equation*}
M_1(h)=h^{r+1}\sum_{k=0}^{+\infty}B_kh^k,\quad 0\le h\ll 1,
\end{equation*}
where
\begin{equation*}
B_k=\tilde b_k b_{0,r}^kc_{0,2r+1} + \tilde b_{k-1} b_{1,r}^{k-1} c_{1,2r+1}+\cdots+ \tilde b_0 b_{k,r}^0 c_{k,2r+1},
\quad k=0,1,\cdots,n.
\end{equation*}
Then the Jacobian matrix
$\frac{\partial(B_0,B_1,\cdots,B_{n})}{\partial(c_{0,2r+1},
\cdots,c_{n,2r+1})}$
is a lower triangle matrix with rank $n+1$. Thus, if $B_i=0$, $i=0,\cdots,n$,
then we get $c_{i,2r+1}=0$, $i=0,\cdots,n$, which implies $M_1(h)\equiv0$.
Then by Theorem \ref{thm3} $M_1(h)$ has at most $n$ positive zeros near $h=0$,
and this upper bound can be reached.

     Next, we study the case of $m=2$. By \eqref{b31m} and (III) of Lemma \ref{lm2}, $M_2(h)$ is given by
     \begin{equation}\label{d0a}
     M_2(h)=\sum_{i=0}^{n+1}A_iI_{i,2r+1}(h)+A_{n+2}L_{n,r+1}(h)=h^{r+1}\sum_{k=0}^{+\infty}B_kh^k.
     \end{equation}
     In order to apply  Theorem \ref{thm3}, we need to get the asymptotic expansion of $L_{n,r+1}(h)$ near $h=0$. Note that by \eqref{z3b} we have $(i+1)b_{i+1,j+1}^{k}-(2j+3)b_{i+2,j}^{k}=-(k+\frac {1}{2})b_{i+1,j+1}^{k}$. Then by \eqref{z3a} for any integers $i\geq 0$ and $j\geq 0$ we obtain
     \begin{equation}\label{d1}
     \begin{aligned}
     L_{i,j+1}(h)=&\,h^{i+j+3}\sum_{k=0}^{\infty}\tilde b_{k}[(i+1)b_{i+1,j+1}^{k}-(2j+3)b_{i+2,j}^{k}]h^k\\
     =&\,-h^{i+j+3}\sum_{k=0}^{\infty}\left(k+\frac{1}{2}\right)\tilde b_{k}b_{i+1,j+1}^{k}h^k.
     \end{aligned}
 \end{equation}
     Then we have
     \begin{equation*}
     L_{n,r+1}(h)=-\frac {2^{r+2}\Gamma(n+\frac{3}{2})\Gamma(r+\frac{5}{2})}{\Gamma(n+r+4)}h^{n+r+3}+O(h^{n+r+4}).
     \end{equation*}
     By \eqref{d0} and \eqref{d1}, 
     the Jacobian matrix
     $\frac{\partial(B_0,B_1,\cdots,B_{n+2})}{\partial(A_0,A_1,,\cdots,A_{n+2})}$
is a lower triangle matrix with rank $n+3$.
     Therefore, by Theorem \ref{thm3} $M_2(h)$ has at most $n+2$ positive zeros near $h=0$, and this upper bound is sharp.

     When $m>2$, for $M_m(h)$ in \eqref{b31m} we suppose
     \begin{equation*}
     \tilde L_{m-1}(h)=\sum_{i=1}^{m-1}A_{n+m-1+i}L_{n-1+i,r+m-i}(h),
     \end{equation*}
     where the coefficients $A_{n+m},\cdots,A_{n+2m-2}$ can  be taken as free parameters. Then by \eqref{d1} $\tilde L_{m-1}(h)$ can  be expanded as
     \begin{equation*}
     \tilde L_{m-1}(h)=-h^{n+m+r+1}\sum_{k=0}^{\infty}\left(k+\frac {1}{2}\right)\tilde b_k B_k h^k,
     \end{equation*}
     where
     \begin{equation}\label{d2}
     \begin{aligned}
     B_k=\sum_{i=1}^{m-1}b_{n+i,r+m-i}^{k}A_{n+m-1+i}.
     \end{aligned}\end{equation}
     By \eqref{b31m}, \eqref{d0} and \eqref{d1}, in order to complete the proof we only need to show
     \begin{equation*}
      \mbox{rank}(D_{m-1})=m-1,
      \end{equation*}
      where
      \begin{equation*}
      D_{m-1}=\frac{\partial(B_0,B_1,\cdots,B_{m-2})}{\partial(A_{n+m},\cdots,A_{n+2m-2})}.
      \end{equation*}

      Then by \eqref{d2} we have
      \begin{equation*}
      D_{m-1}=
      \begin{pmatrix}
      b_{n+1,r+m-1}^{0}&b_{n+2,r+m-2}^{0}&\cdots &b_{n+m-1,r+1}^{0}\\
      b_{n+1,r+m-1}^{1}&b_{n+2,r+m-2}^{1}&\cdots &b_{n+m-1,r+1}^{1}\\
      \vdots &\vdots &\ddots &\vdots \\
      b_{n+1,r+m-1}^{m-2}&b_{n+2,r+m-2}^{m-2}&\cdots &b_{n+m-1,r+1}^{m-2}
      \end{pmatrix}.
      \end{equation*}
      By \eqref{z3b} we can remove common factors from each row and each column of $D_{m-1}$, and then get
      \begin{equation}\label{d3}
      \tilde D_{m-1}=
      \begin{pmatrix}
      1&1&1&\cdots &1\\
      n+\frac {3}{2}&n+\frac {5}{2}&n+\frac {7}{2}&\cdots &n+m-\frac {1}{2}\\
      \prod\limits_{k=n}^{n+1}(k\!+\!\frac {3}{2})&\prod\limits_{k=n}^{n+1}(k\!+\!\frac {5}{2})&\prod\limits_{k=n}^{n+1}(k\!+\!\frac {7}{2})&\cdots &\prod\limits_{k=n}^{n+1}(k\!+\!m\!-\!\frac {1}{2})\\
      \vdots &\vdots &\vdots &\ddots &\vdots \\
      \prod\limits_{k=n}^{n+m-3}\!\!\!(k\!+\!\frac {3}{2})&\prod\limits_{k=n}^{n+m-3}\!\!\!(k\!+\!\frac {5}{2})&\prod\limits_{k=n}^{n+m-3}\!\!\!(k\!+\!\frac {7}{2})&\cdots &\prod\limits_{k=n}^{n+m-3}\!\!(k\!+\!m\!-\!\frac {1}{2})
      \end{pmatrix}.
      \end{equation}
      Then rank$(D_{m-1})=m-1$ if and only if rank$(\tilde D_{m-1})=m-1$.

      We add the $i$th row multiplied by $-(n+i+\frac {1}{2})$ to the $(i+1)$th row in the order for $i$ from $m-2$ to 1, and then obtain
      \begin{equation*}
      \tilde D_{m-1}\longrightarrow
      \begin{pmatrix}
      1&E_{1\times (m-2)}\\
      0&D_{m-2}^{\ast}
      \end{pmatrix},
      \end{equation*}
      where $E_{1\times (m-2)}=(1,1,\cdots,1)_{1\times (m-2)}$ and
      \begin{equation*}
      D_{m-2}^{\ast}=
      \begin{pmatrix}
      1&2&\cdots &m-2\\
      n\!+\!\frac {5}{2}&2(n\!+\!\frac {7}{2})&\cdots &(m\!-\!2)(n\!+\!m\!-\!\frac {1}{2})\\
      \vdots &\vdots &\ddots &\vdots \\
      \prod\limits_{k=n}^{n+m-4}\!(k\!+\!\frac {5}{2})&2\!\!\prod\limits_{k=n}^{n+m-4}(k\!+\!\frac {7}{2})&\cdots &(m-2)\!\!\prod\limits_{k=n}^{n+m-4}(k\!+\!m\!-\!\frac {1}{2})
      \end{pmatrix}.
      \end{equation*}
By dividing the $i$th column of $D_{m-2}^{\ast}$ by $i$, $i=2,\cdots ,m-2$,
      and replacing $n$ by $n^{\ast}-1$, we get $\tilde D_{m-2}$ with $n=n^{\ast}$. Then rank$(\tilde D_{m-1})=m-1$ if and only if
      rank$(\tilde D_{m-2})=m-2$ for any $n$. It is evident that by \eqref{d3} rank$(\tilde D_{2})=2$ for any $n$.
      Therefore, rank$(\tilde D_{m-1})=m-1$ for any $m\geq 3$. The proof is completed.
    \end{proof}

Next, we shall prove Theorem \ref{thm2} by studying the asymptotic expansion of $\widetilde M(h)$ in \eqref{t6} near $h=0$.
To simplify the Melnikov function $\widetilde M(h)$ of system \eqref{b12},
we have the next lemma for the integral
\begin{equation}\label{f1}
\widetilde M_l(h)= \int_{\Gamma_h^+}\sum_{k=\tilde r}^{l+\tilde r-1} Q^e_{n,2k}(x)y^{2k}\mathrm{d}x=\sum_{k=\tilde r}^{l+\tilde r-1}\sum_{i=0}^n b_{i,2k}J_{i,2k},
\end{equation}
where $\tilde r=[\frac{s_1+1}{2}]$ and $l=[\frac{\hat s-2\tilde r+2}{2}]$. 

\begin{lm}\label{lm8}
    For any integers $n\ge1$, $\tilde r\ge1$ and $l\ge 2$,
    $\widetilde M_{l}(h)$ in \eqref{f1} can be rewritten in the following form
    \begin{equation}\label{f2}
    \begin{aligned}
    \widetilde M_{l}(h)&=\sum_{i=0}^{n+l-1} A_{i}\,J_{i,2\tilde r}+\sum_{i=1}^{l-1} A_{n+l-1+i}\, \widetilde L_{n-1+i,\tilde r+l-1-i}(h),
    \end{aligned}\end{equation}
    where the coefficients $A_{i}$, $i=0,\cdots, n+2(l-1)$ can be taken as free parameters, and
    $\widetilde L_{i,j}(h)=(i+1)J_{i+1,2j+2}-(2j+2)J_{i+2,2j}$.
    \end{lm}

    \begin{proof}
Using induction on $l$, we can prove \eqref{f2} by the same way
as in the proof of Lemma \ref{lm6}.
We shall give the proof for the case of $l=2$.
For $l>2$ the details are omitted here.

For $\widetilde M_2(h)$ in \eqref{f1} we have
    \begin{equation}\label{f3}
    \begin{aligned}
    \widetilde M_2(h)&
    &=\sum_{i=0}^{n} b_{i,2\tilde r}J_{i,2\tilde r}
    +\sum_{i=0}^{n} b_{i,2\tilde r+2}J_{i,2\tilde r+2}.
    \end{aligned}\end{equation}
    For any $0\leq i \leq n-2$, by \eqref{zg} we get that
    $J_{i,2\tilde r+2}$ can be expressed as a linear combination of $J_{i+1,2\tilde r}$,
    $J_{i+2,2\tilde r}$, $\cdots$, $J_{n,2\tilde r}$ and $J_{n-1,2\tilde r+2}$.
Then $\widetilde M_2(h)$ in \eqref{f3} can be written as
    \begin{equation*}
    \widetilde M_2(h)=\sum_{i=0}^{n}\tilde c_{i} J_{i,2\tilde r}+\tilde c_{n+1} J_{n-1,2\tilde r+2}+\tilde c_{n+2} J_{n,2\tilde r+2},
    \end{equation*}
    where the coefficients $\tilde c_{0},\tilde c_{1},\cdots,\tilde c_{n+2}$ can be taken as independent parameters. Using \eqref{zf} with $i=n-1$ and $i=n$ in order,
     we can have \eqref{f2} for $l=2$ with
    \begin{equation}\label{f5}
    \begin{aligned}
    &A_{i}=\tilde c_{i},\quad 0\leq i \leq n-1,\\
    &A_{n}=\tilde c_{n}+\frac {2(2\tilde r+2)}{2n-1}\tilde c_{n+1},\\
    &A_{n+1}=-\frac {2\tilde r+2}{4n^2-1}\tilde c_{n+1}+\frac {2(2\tilde r+2)}{2n+1}\tilde c_{n+2},\\
    &A_{n+2}=\frac {n}{4n^2-1}\tilde c_{n+1}+\frac {1}{2n+1}\tilde c_{n+2}.
    \end{aligned}\end{equation}
    By \eqref{f5} we can get that the coefficients $A_{i},0\leq i \leq n+2$ are linearly independent.
    \end{proof}

Next, we present the proof of Theorem \ref{thm2} as below.

\begin{proof}[Proof of Theorem \ref{thm2}]
We shall prove the conclusion for $n\ge1$ and $\hat s>s_1$ by using Theorem \ref{thm4}.
The other cases can be similarly proved.

By \eqref{t6}, $\widetilde M(h)$  can be written as
\begin{equation}\label{f6}
\widetilde M(h)= \sum_{s=s_1}^{\hat s}\sum_{i=0}^n b_{i,s}J_{i,s}(h)
=\widehat M_m(h)+\widetilde M_l(h),
\end{equation}
where
\begin{equation}\label{f7}
\widehat M_m(h)= \frac{1}{2} \sum_{k=r}^{m+r-1}\sum_{i=0}^n b_{i,2k+1} I_{i,2i+1}(h),\quad
\widetilde M_l(h)=\sum_{k=\tilde r}^{l+\tilde r-1}\sum_{i=0}^n b_{i,2k} J_{i,2k}(h),
\end{equation}
and $r=[\frac{s_1}{2}]$, $m=[\frac{\hat s-2r+1}{2}]$, $\tilde r=[\frac{s_1+1}{2}]$
and $l=[\frac{\hat s-2\tilde r+2}{2}]$.
By \eqref{z3a} and \eqref{zh1}, $\widehat M_m(h)$ and $\widetilde M_l(h)$ can be expanded as
\begin{equation}\label{f8}
\widehat M_m(h)=h^{r+1}\sum_{k=0}^{+\infty}B_kh^k,\quad
\widetilde M_l(h)=h^{\tilde r+\frac{1}{2}}\sum_{k=0}^{+\infty}\widetilde B_k h^k.
\end{equation}
Let  $\boldsymbol{\delta}_s=(b_{0,s},b_{1,s},\cdots,b_{n,s})$, $s_1\le s \le \hat s$.
Note that by \eqref{f7} and \eqref{f8} the coefficients $B_k$'s and $\widetilde B_k$'s
in \eqref{f8} linearly depend on parameter vectors $\boldsymbol{\delta}_{2k+1}$'s
and $\boldsymbol{\delta}_{2k}$'s, respectively.
From the proof of Theorem \ref{thm1} we get that $\widehat M_m(h)\equiv0$ when $B_k=0$,
$0\le k\le n+2m-2$, and
\begin{equation*}
\mbox{rank}\left(\frac{\partial(B_0,B_1,\cdots,B_{n+2m-2})}
{\partial(\boldsymbol{\delta}_{2r+1}, \boldsymbol{\delta}_{2r+3},\cdots,
\boldsymbol{\delta}_{2r+2m-1})}\right)=n+2m-1.
\end{equation*}
Then by \eqref{f6}-\eqref{f8} and Theorem \ref{thm4} we only need to prove that for any integer $l\ge1$
\begin{equation}\label{f9}
\begin{split}
&\mbox{rank}\left(\frac{\partial(\widetilde B_0,\widetilde B_1,\cdots,\widetilde B_{n+2l-2})}
{\partial(\boldsymbol{\delta}_{2\tilde r}, \boldsymbol{\delta}_{2\tilde r+2},\cdots,
\boldsymbol{\delta}_{2\tilde r+2l-2})}\right)=n+2l-1\\
\mbox{and}\,\,& \widetilde M_l(h)\equiv0 \,\,\mbox{if}\,\, \widetilde B_k=0,\,\,
0\le k\le n+2l-2.\\
\end{split}
\end{equation}

     For the case of $l=1$, from \eqref{f7} we have
     \begin{equation}\label{f10}
     \begin{aligned}
     \widetilde M_1(h)=\int_{\Gamma_h^+}\sum_{i=0}^{n}b_{i,2\tilde r}(1-\cos(x))^iy^{2\tilde r}\mathrm{d}x=\sum_{i=0}^{n}b_{i,2\tilde r}J_{i,2\tilde r}.
    \end{aligned}\end{equation}
By (III) of Lemma \ref{lm4}, we get
\begin{equation}\label{f11}
        J_{i,2\tilde r}(h)=h^{i+\tilde r+\frac {1}{2}}\sum_{k=0}^{+\infty}\tilde b_k c_{i,\tilde r}^k h^k,\quad 0\le h\ll 1,
    \end{equation}
    where $\tilde b_k$ and $c_{i,r}^k$ are given in \eqref{z3b} and \eqref{zh2}.
Then by \eqref{f8}$|_{l=1}$, \eqref{f10} and \eqref{f11} we obtain
\begin{equation*}
B_i=\tilde b_k c_{0,\tilde r}^kb_{0,2\tilde r} + \tilde b_{k-1} b_{1,\tilde r}^{k-1} b_{1,2\tilde r}+\cdots+ \tilde b_0 c_{k,\tilde r}^0 b_{k,2\tilde r},
\quad i=0,1,\cdots,n,
\end{equation*}
which implies that the Jacobian matrix
$\frac{\partial(\widetilde{B}_0,\widetilde{B}_1,\cdots,\widetilde{B}_{n})}
{\partial(b_{0,2\tilde r},
\cdots,b_{n,2\tilde r})}$
is a lower triangle matrix with rank $n+1$,
and $\widetilde B_k=0$,
$0\le k\le n$ if and only if $b_{i,2\tilde r}=0$, $0\le i\le n$.
Then \eqref{f9} holds for $l=1$.

     Next, we study the case of $l=2$. By \eqref{f2}, $\widetilde  M_2(h)$ is given by
     \begin{equation}\label{g0}
     \widetilde M_2(h)=\sum_{i=0}^{n+1}A_iJ_{i,2\tilde r}(h)+A_{n+2}\widetilde L_{n,\tilde r+1}(h).
     \end{equation}
     We need to get the series expansion of $\widetilde L_{n,\tilde r+1}(h)$ near $h=0$. Note that by \eqref{zh2} we have $(i+1)c_{i+1,j+1}^{k}-(2j+2)c_{i+2,j}^{k}=-(k+\frac {1}{2})c_{i+1,j+1}^{k}$. Then by \eqref{zh1} for any integers $i\geq 0$ and $j\geq 1$ we obtain
     \begin{equation}\label{g1}
     \begin{aligned}
     \widetilde L_{i,j}(h)&=h^{i+j+\frac {5}{2}}\sum_{k=0}^{\infty}\tilde b_{k}[(i+1)c_{i+1,j+1}^{k}-(2j+2)c_{i+2,j}^{k}]h^k\\
     &=-h^{i+j+\frac {5}{2}}\sum_{k=0}^{\infty}(k+\frac {1}{2})\tilde b_{k}c_{i+1,j+1}^{k}h^k,
     \end{aligned}\end{equation}
     which yields
     \begin{equation}\label{g1a}
     \widetilde L_{n,\tilde r}(h)=-\frac {2^{\tilde r+\frac{1}{2}}\Gamma(n+\frac{3}{2})\Gamma(\tilde r+2)}{\Gamma(n+\tilde r+\frac{7}{2})}h^{n+\tilde r+\frac{5}{2}}+O(h^{n+\tilde r+\frac{7}{2}}).
     \end{equation}
     Then by \eqref{f8}$|_{l=2}$, \eqref{f11}, \eqref{g0} and \eqref{g1a} we get
 that
$\frac{\partial(\widetilde{B}_0,\widetilde{B}_1,\cdots,\widetilde{B}_{n+2})}
{\partial(A_0,A_1,
\cdots,A_{n+2})}$
is a lower triangle matrix with rank $n+3$,
and $\widetilde B_k=0$,
$0\le k\le n+2$ if and only if $A_{i}=0$, $0\le i\le n+2$.
Because the coefficients $A_i$'s, $0\le i\le n+2$ in \eqref{g0} are independent linear combinations of $\boldsymbol{\delta}_{2\tilde r}$ and $\boldsymbol{\delta}_{2\tilde r+2}$.
Then \eqref{f9} holds for $l=2$.

     For $l\ge 3$, $\widetilde M_{l}(h)$ in \eqref{f8} can be still written
     in the form of \eqref{f2}. For \eqref{f2} we suppose
     \begin{equation*}
     \widehat L_{l-1}(h)=\sum_{i=1}^{l-1}A_{n+l-1+i}\widetilde L_{n-1+i,\tilde r+l-1-i}(h),
     \end{equation*}
     where the coefficients $A_{n+l},\cdots,A_{n+2l-2}$ can  be taken as free parameters. Then by \eqref{g1} $\widehat L_{l-1}(h)$ can  be expanded as
     \begin{equation*}
     \widehat L_{l-1}(h)=-h^{n+l+\tilde r+\frac {1}{2}}\sum_{k=0}^{\infty}(k+\frac {1}{2})\tilde b_k \widehat{B}_k h^k,
     \end{equation*}
     where
     \begin{equation}\label{g2}
     \begin{aligned}
     \widehat{B}_k=\sum_{i=1}^{l-1}c_{n+i,\tilde r+l-i}^{k}A_{n+l-1+i}.
     \end{aligned}\end{equation}
     From the above proof we can see that we only need to show
     $\mbox{rank}(D_{l-1})=l-1$, where
      \begin{equation*}
      D_{l-1}=\frac{\partial(\widehat{B}_0,\widehat{B}_1,\cdots,\widehat{B}_{l-2})}{\partial(A_{n+l},\cdots,A_{n+2l-2})}.
      \end{equation*}

      By \eqref{g2} we have
      \begin{equation*}
      D_{l-1}=
      \begin{pmatrix}
      c_{n+1,\tilde r+l-1}^{0}&c_{n+2,\tilde r+l-2}^{0}&\cdots &c_{n+l-1,\tilde r+1}^{0}\\
      c_{n+1,\tilde r+l-1}^{1}&c_{n+2,\tilde r+l-2}^{1}&\cdots &c_{n+l-1,\tilde r+1}^{1}\\
      \vdots &\vdots &\ddots &\vdots \\
      c_{n+1,\tilde r+l-1}^{l-2}&c_{n+2,\tilde r+l-2}^{l-2}&\cdots &c_{n+l-1,\tilde r+1}^{l-2}
      \end{pmatrix}.
      \end{equation*}
      Using \eqref{zh2} we can remove common factors from each row and column of $D_{l-1}$,
      and then get
      \begin{equation}\label{g3}
      \widetilde D_{l-1}=
      \begin{pmatrix}
      1&1&1&\cdots &1\\
      n+\frac {3}{2}&n+\frac {5}{2}&n+\frac {7}{2}&\cdots &n+l-\frac {1}{2}\\
      \prod\limits_{k=n}^{n+1}(k\!+\!\frac {3}{2})&\prod\limits_{k=n}^{n+1}(k\!+\!\frac {5}{2})&\prod\limits_{k=n}^{n+1}(k\!+\!\frac {7}{2})&\cdots &
      \prod\limits_{k=n}^{n+1}(k\!+\!l\!-\!\frac {1}{2})\\
      \vdots &\vdots &\vdots &\ddots &\vdots \\
      \prod\limits_{k=n}^{n+l-3}(k\!+\!\frac {3}{2})&\prod\limits_{k=n}^{n+l-3}(k\!+\!\frac {5}{2})&\prod\limits_{k=n}^{n+l-3}(k\!+\!\frac {7}{2})&\cdots &\prod\limits_{k=n}^{n+l-3}(k\!+\!l\!-\!\frac {1}{2})
      \end{pmatrix}.
      \end{equation}
      Then rank$(D_{l-1})=l-1$ if and only if rank$(\widetilde D_{l-1})=l-1$.

      We add the $i$th row multiplied by $-(n+i+\frac {1}{2})$ to the $(i+1)$th row in the order from $l-2$ to 1, and then obtain
      \begin{equation*}
      \widetilde D_{l-1}\longrightarrow
      \begin{pmatrix}
      1&E_{1\times (l-2)}\\
      0&D_{l-2}^{\ast}
      \end{pmatrix},
      \end{equation*}
      where $E_{1\times (l-2)}=(1,1,\cdots,1)_{1\times (l-2)}$ and
      \begin{equation*}
      D_{l-2}^{\ast}=
      \begin{pmatrix}
      1&2&\cdots &l-2\\
      n\!+\!\frac {5}{2}&2(n\!+\!\frac {7}{2})&\cdots &(l\!-\!2)(n\!+\!l\!-\!\frac {1}{2})\\
      \vdots &\vdots &\ddots &\vdots \\
      \prod\limits_{k=n}^{n+l-4}(k\!+\!\frac {5}{2})&2\prod\limits_{k=n}^{n+l-4}(k\!+\!\frac {7}{2})&\cdots &(l\!-\!2)\prod\limits_{k=n}^{n+l-4}(k\!+\!l\!-\!\frac {1}{2})
      \end{pmatrix}.
      \end{equation*}
      Dividing the $i$th column of $D_{l-2}^{\ast}$ by $i$, $i=2,3,\cdots ,l-2$,
       from $D_{l-2}^{\ast}$ we get $\widetilde D_{l-2}$ given by \eqref{g3}$|_{n=n+1}$.
      Then rank$(\widetilde D_{l-1})=l-1$ if and only if rank$(\widetilde D_{l-2})=l-2$ for any $n$.
      It is evident that by \eqref{g3} rank$(\widetilde D_{2})=2$ for any $n$. Therefore, rank $(\widetilde D_{l-1})=l-1$ for any $l\geq 3$. The proof is completed.
    \end{proof}

\section{Acknowledgments}
This work was supported by the National Natural Science Foundation of China (NSFC No. 12371175, 11871042).


\end{document}